\documentclass{article}

\usepackage{microtype}
\usepackage{graphicx}
\usepackage{subfigure}
\usepackage{booktabs} 

\usepackage{mathtools,amssymb,amsthm,color}

\allowdisplaybreaks

\usepackage{hyperref}

\usepackage{amsthm}
\newtheorem{theorem}{Theorem}
\newtheorem{lemma}{Lemma}
\newtheorem{prop}{Proposition}

\newtheorem{conjinner}{Conjecture}
\newenvironment{conj}[1]{%
  \renewcommand\theconjinner{#1}%
  \conjinner
}{\endconjinner}

\DeclareMathOperator{\argmin}{argmin}
\DeclareMathOperator{\degree}{deg}
\DeclareMathOperator{\tr}{tr}
\DeclareMathOperator{\diag}{diag}
\DeclareMathOperator{\spn}{span}

\newcommand{\tp}{{\scriptscriptstyle\mathsf{T}}}
\newcommand{\pp}{{\scriptscriptstyle++}}

\usepackage[accepted]{icml2020}

\icmltitlerunning{Noncommutative Arithmetic-Geometric Mean Conjecture is False}

\begin{document}

\twocolumn[
\icmltitle{Recht--R\'e Noncommutative Arithmetic-Geometric Mean Conjecture is False}



\icmlsetsymbol{equal}{*}

\begin{icmlauthorlist}
\icmlauthor{Zehua Lai}{UC}
\icmlauthor{Lek-Heng Lim}{UC}
\end{icmlauthorlist}

\icmlaffiliation{UC}{Computational and Applied Mathematics Initiative,
University of Chicago, Chicago, IL 60637, USA}

\icmlcorrespondingauthor{Zehua Lai}{laizehua@uchicago.edu}

\icmlkeywords{Positivstellensatz, noncommutative arithmetic-geometric mean conjecture}

\vskip 0.3in
]



\printAffiliationsAndNotice{\icmlEqualContribution} 

\begin{abstract}
Stochastic optimization algorithms have become indispensable in modern machine
learning. An unresolved foundational question in this area is the
difference between with-replacement sampling and without-replacement
sampling --- does the latter have superior convergence rate compared to the
former? A groundbreaking result of Recht and R\'e reduces the problem to a
noncommutative analogue of the arithmetic-geometric mean inequality where $n$
positive numbers are replaced by $n$ positive definite matrices. If
this inequality holds for all $n$, then without-replacement sampling indeed
outperforms with-replacement sampling. The conjectured Recht--R\'e
inequality has so far only been established for $n = 2$ and a special case
of $n = 3$. We will show that the Recht--R\'e conjecture is false for general
$n$. Our approach relies on the noncommutative Positivstellensatz, which
allows us to reduce the conjectured inequality to a semidefinite program
and the validity of the conjecture to certain bounds for the optimum
values, which we show are false as soon as $n = 5$.
\end{abstract}

\section{Introduction}

The breathtaking reach of deep learning, permeating every area of science and technology, has led to an outsize role for randomized optimization algorithms. It is probably fair to say that in the absence of randomized algorithms, deep learning would not have achieved its spectacular level of success. Fitting an exceedingly high-dimensional model with an exceedingly large training set would have been prohibitively expensive without some form of \emph{random sampling}, which in addition provides other crucial benefits such as saddle-point avoidance \cite{fang19,jin17}. As such, in machine learning computations, stochastic variants of gradient descent \cite{Bottou2010, johnson2013, Nemirovski2008}, alternating projections \cite{Strohmer2009}, coordinate descent \cite{Nesterov2012}, and other algorithms have largely overtaken their classical deterministic counterparts in  relevance and utility.

There are numerous random sampling strategies but the most fundamental question, before all other considerations, is deciding between \emph{sampling with replacement} or \emph{sampling without replacement}. In the vast majority of randomized algorithms, a random sample  is selected or a random action is performed \emph{with replacement} from a pool, making the randomness in each iteration independent and thus easier (often much easier) to analyze. However, when it comes to practical realizations of these algorithms, one invariably sample \emph{without replacement}, since they are easier (often much easier) to implement. Take the ubiquitous stochastic gradient descent for example, many if not most implementations would pass through each item exactly once in a random order --- this is sampling without replacement. Likewise, in implementations of randomized coordinate descent, coordinates are usually just chosen in a random order  --- again sampling without replacement. 

Apart from its ease of implementation, there are other reasons for favoring  without-replacement sampling. Empirical evidence \cite{bottou2009} suggests that in stochastic gradient descent,  without-replacement sampling regularly outperforms with-replacement sampling.  Theoretical results also point towards without-replacement sampling: Under standard convexity assumptions, the convergence rate of a without-replacement sampling algorithm typically beats a with-replacement sampling one by a factor of $O(n^{-1/2})$ or $O(n^{-1})$. This has been established for stochastic gradient descent \cite{Shamir2016, nagaraj2019} and for coordinate descent \cite{Beck2013, Wright2015}.

\citet{recht2012toward} proposed a matrix theoretic approach to compare the efficacy of with- and without-replacement sampling methods. Since nearly every common optimization algorithm, deterministic or randomized, works with a linear or quadratic approximation of the objective function locally, it suffices to examine the two sampling strategies on linear or quadratic functions to understand their local convergence behaviors. In this case, the iteration reduces to matrix multiplication and both sampling procedures are linearly convergent (often called ``exponential convergence'' in machine learning). The question of which is better then reduces to comparing their linear convergence rates. In this context, \citet{recht2012toward} showed that without-replacement sampling outperforms with-replacement sampling provided the following noncommutative version of the arithmetic-geometric mean inequality holds. 
\begin{conj}{1}[\citealt{recht2012toward}]\label{conj1}
Let $n$ be a positive integer, $A_1,\dots, A_n$ be symmetric positive semidefinite matrices, and $\lVert\,\cdot\,\rVert$ be the spectral norm. Then for any $m \le n$,
\begin{multline}\label{eq:conj1}
\frac{1}{n^m}\biggl\lVert\sum_{1\leq j_1,\dots, j_m \leq n}\hspace*{-4ex} A_{j_1}\cdots A_{j_m}\biggr\rVert \geq \\
\frac{(n-m)!}{n!} \biggl\lVert\sum_{\substack{1\leq j_1,\dots, j_m \leq n, \\j_1,\dots, j_m \; \text{distinct}}}\hspace*{-4ex} A_{j_1}\cdots A_{j_m}\biggr\rVert.
\end{multline}
\end{conj}
While one may also ask if \eqref{eq:conj1} holds for other norms, the most natural and basic choice is the spectral norm, i.e., the operator $2$-norm. Unless specified otherwise, $\lVert\,\cdot\,\rVert$ will always denote the spectral norm in this article.

To give an inkling of how \eqref{eq:conj1} arises, consider the Kaczmarz algorithm \cite{Strohmer2009} where we attempt to solve an overdetermined linear system $Cx = b$, $C\in \mathbb{R}^{p \times d}$, $p > d$, with $i$th row vector\footnote{We adopt standard convention that any vector $x\in\mathbb{R}^d$ is a column vector; a row vector will always be denoted $x^\tp$.} $c_i^\tp$ where $c_i \in \mathbb{R}^d$. For $k = 1,2,\dots,$ the $(k+1)$th iterate is formed with a randomly chosen $i$ and 
\[
x^{(k+1)} = x^{(k)}+ \frac{b_i-\langle c_i, x^{(k)}\rangle}{\lVert c_i\rVert^2}c_i.
\] 
The $k$th error $e^{(k)} = x^{(k)} - x^*$ is then
\[
e^{(k+1)} = \biggl(I-\frac{c_i c_i^\tp }{\lVert c_i\rVert^2}\biggr)e^{(k)} \eqqcolon P_{c_i} e^{(k)},
\]
where $P_{c} \in \mathbb{R}^{d \times d}$, the orthogonal projector onto $\spn\{c\}^\perp$, is clearly symmetric positive semidefinite. A careful analysis would show that the relative efficacy of with- and without-replacement sampling depends on a multitude of inequalities like $\lVert A^4+B^4+AB^2A+BA^2B\rVert \ge 2\lVert AB^2A+BA^2B\rVert$, which are difficult to analyze on a case-by-case basis. Nevertheless, more general heuristics \cite{recht2012toward} would lead to \eqref{eq:conj1} --- if it holds, then without-replacement sampling is expected to outperform with-replacement sampling. In fact, the gap  can be significant --- for random Wishart matrices, \citet{recht2012toward}  showed that the ratio between the two sides of \eqref{eq:conj1} increases exponentially with $m$.

\paragraph{Current status:} \citet{recht2012toward} applied results from random matrix theory to show that Conjecture~\ref{conj1} holds with high probability for (i) independent Wishart matrices, and (ii) the incremental gradient method. To date, extensive numerical simulations have produced no counterexample.
Conjecture~\ref{conj1} has been rigorously established only in very special cases, notably for $(m, n) = (2,2)$  \cite{recht2012toward} and $(m,n) = (3,3k)$ \citep{zhang2018note}.

\paragraph{Our contributions:}  We show how to  transform Conjecture~\ref{conj1}  into a form where the \emph{noncommutative Positivstellensatz} applies, which implies in particular that for any specific values of $m$ and $n$, the conjecture can be checked via two semidefinite programs. This allows us to show in Section~\ref{sec3} that the conjecture is false as soon as $m=n = 5$. We also establish in Section~\ref{sec2} that the conjecture holds for $m = 2$ and $3$ with arbitrary  $n$ by extending the approach in \citet{zhang2018note}. While the conjectured inequality \eqref{eq:conj1} is clearly sharp (as we may choose all $A_i$'s to be equal) whenever it is true, we show in Section~\ref{sec4} that the $m=2$ case may nonetheless be improved in a different sense, and we do likewise for $m = 3$ in Section~\ref{sec2}.  The $m = 4$ case remains open but our  noncommutative Positivstellensatz approach permits us to at least check that it holds for $n = 4$ and $5$ in Section~\ref{sec3}.

Over the next two sections, we will transform Recht and R\'e's Conjecture~\ref{conj1} into a ``Loewner form'' (Conjecture~\ref{conjL}), a ``sum-of-squares form'' (Conjecture~\ref{conjSOS}), and finally a ``semidefinite program form'' (Conjecture~\ref{conjSDP}). All four conjectures are equivalent but the correctness of the last one for any $m,n$ can be readily checked as a semidefinite program.

\section{Recht--R\'e inequality for $m=2$ and $3$}\label{sec2}

Our goal here is to establish \eqref{eq:conj1} for a pair and a triple of matrices. In so doing, we take Conjecture~\ref{conj1} a step closer to a form where noncommutative Positivstellensatz applies. There is independent value in establishing these two special cases given that the classical noncommutative arithmetic-geometric-harmonic mean inequality \cite{intell} is only known for a pair of matrices but nonetheless attracted a lot of interests from linear algebraists. These special cases also have implications on randomized algorithms --- take the Kaczmarz algorithm for example, the fact that Conjecture~\ref{conj1} holds for $m=2$ and $3$ implies that if we randomly choose two or three distinct samples, perform the iterations, and sample again, then this ``replacing after every two or three samples'' strategy will converge faster than a ``replacing after every sample'' strategy.

We begin by providing some context for the inequality \eqref{eq:conj1}.  The usual arithmetic-geometric mean inequality for $n$ nonnegative real numbers $a_1, \dots , a_n$, i.e.,
\[
(a_1 + \dots + a_n)/n \geq (a_1  \cdots a_n)^{1/n},
\]
is a special case of Maclaurin’s inequality \cite{hardy1952inequalities}: If we define
\[
s_m \coloneqq \frac{1}{\binom{n}{m}} \sum_{1\leq j_1 < \dots < j_m \leq n}\hspace*{-4ex} a_{j_1} \cdots a_{j_m},
\]
then $s_1 \geq \sqrt{s_2} \geq \dots \geq \sqrt[n]{s_n}$. So $s_1 \ge \sqrt[m]{s_m}$ gives us
\[
\frac{1}{n^m} (a_1 + \dots + a_n)^m \ge
\frac{(n-m)!}{n!} \hspace*{-1ex}\sum_{\substack{1\leq j_1,\dots, j_m \leq n, \\j_1,\dots, j_m \; \text{distinct}}}\hspace*{-4ex} a_{j_1} \cdots a_{j_m},
\]
which is just \eqref{eq:conj1} for $1$-by-$1$ positive semidefinite matrices.

For real symmetric or complex Hermitian  matrices $A, B$, the Loewner order is defined by $A\succeq B$ iff $A-B$ is positive semidefinite. The Maclaurin’s inequality has several noncommutative extensions but we regard the following as the starting point for all noncommutative  arithmetic-geometric mean inequalities.
\begin{prop} \label{propmn2}
For any unitary invariant norm $\lVert\, \cdot\, \rVert$ and Hermitian matrices $A, B$, $\lVert AB+BA\rVert \leq \lVert A^2+B^2\rVert$ and $2\lVert AB+BA\rVert \leq \lVert (A+B)^2\rVert$.
\end{prop}
\begin{proof}
Since $-A^2 -B^2 \preceq AB+BA \preceq A^2 + B^2$, by Lemma~2.1 in \citet{bhatia2008matrix}, the desired inequalities hold for any unitary invariant norm. 
\end{proof}
The result was extended  to compact operators on a separable Hilbert space and strengthened to $2\lVert A^*B\rVert \leq \lVert A^*A + B^*B\rVert$ in \citet{bhatia1990singular}, with yet other extensions in \citet{bhatia2008matrix, bhatia2000notes}. In \citet{recht2012toward}, Conjecture~\ref{conj1} was also formulated as an extension of Proposition~\ref{propmn2}, with the second inequality corresponding to the $m=n=2$ case.

Straightforward counterexamples for $n = 3$ show  that we cannot simply drop the norm in \eqref{eq:conj1} and replace the inequality $\ge$ with the Loewner order $\succeq$. Nevertheless Conjecture~\ref{conj1} may  be written as \emph{two} Loewner inequalities, as demonstrated by \citet{zhang2018note}.
\begin{conj}{1A}[Loewner  form]\label{conjL}
Let $A_1,\dots, A_n$ be symmetric positive semidefinite and $A_1+\dots + A_n \preceq n I$. Then for any $m \le n$,
\begin{equation}\label{eq:conjL}
-\frac{n!}{(n-m)!} I \preceq \sum_{\substack{1\leq j_1,\dots, j_m \leq n, \\j_1,\dots, j_m \; \text{distinct}}}\hspace*{-4ex} A_{j_1}\cdots A_{j_m} \preceq  \frac{n!}{(n-m)!} I.
\end{equation}
\end{conj}
We prefer this equivalent formulation \eqref{eq:conjL} as the original formulation \eqref{eq:conj1} hides an asymmetry --- note that there is an upper bound and a lower bound in \eqref{eq:conjL} and there is no reason to expect that they should have the same magnitude. In fact, as we will see in the later sections, the upper and lower bounds have different magnitudes in every case that we examined.

We will next prove Conjecture~\ref{conj1} in its equivalent form Conjecture~\ref{conjL} for $m = 2$ and $3$. Our proofs rely on techniques introduced by \citet{zhang2018note} in his proof for the case $m = 3$, $n = 3k$, but our two additional contributions are that (i) we will obtain better lower bounds (deferred to  Section~\ref{sec4}), and (ii) our proof will work for arbitrary $n$ (not necessarily a multiple of $3$).
\begin{theorem}[Recht--R\'e for $m=2$]\label{theorem1} 
Let $A_1,\dots, A_n$ be symmetric positive semidefinite and $A_1+\dots + A_n \preceq nI$. Then
\begin{equation}\label{eq:m2}
-n(n-1)I \preceq \sum_{i\neq j}A_iA_j \preceq n(n-1)I.
\end{equation}
\end{theorem}
\begin{proof}
The right inequality in \eqref{eq:m2} follows from
\begin{align*}
(n&-1)\sum_{i,j}A_iA_j- n\sum_{i\neq j}A_iA_j\\
&= (n-1)\sum_i A_i^2 - \sum_{i\neq j}A_iA_j
= \sum_{i<j}(A_i-A_j)^2 \succeq 0,
\end{align*}
and so
\[
\sum_{i\neq j}A_iA_j \preceq \frac{n-1}{n}\sum_{i,j}A_iA_j \preceq n(n-1)I.
\]
For the left inequality in \eqref{eq:m2}, expand $(\sum_i A_i)^2\succeq 0$ to get
\[
\sum_i A_i^2 \succeq -\sum_{i\neq j}A_iA_j.
\]
Let $B \coloneqq nI-\sum_i A_i \succeq 0$ and $B_i \coloneqq  A_i+\frac{1}{n}B \succeq 0$. So $\sum_i B_i = nI$. Then
\begin{align*}
-n\sum_{i\neq j}A_iA_j &= -(n-1)\sum_{i\neq j}A_iA_j-\sum_{i\neq j}A_iA_j\\
&\preceq (n-1)\sum_i A_i^2 -\sum_{i\neq j}A_iA_j\\
&= \sum_{i<j}(A_i-A_j)^2
= \sum_{i<j}(B_i-B_j)^2\\
&= (n-1)\sum_i B_i^2 -\sum_{i\neq j}B_iB_j\\
&= n\sum_i B_i^2 - \Bigl(\sum_i B_i\Bigr)^2\\
&= n\sum_i B_i^2 - n^2I.
\end{align*}
Therefore
\[
-\sum_{i\neq j}A_iA_j-(n-1)nI \preceq \sum_i B_i^2 - n^2I =\sum_i (B_i^2-n B_i).
\]
The eigenvalues of $B_i$ fall between $0$ and $n$, so the eigenvalues of $B_i^2-n B_i$ are all nonpositive, i.e., $B_i^2- nB_i \preceq 0$. Hence $-\sum_{i\neq j}A_iA_j-(n-1)nI\preceq 0$.
\end{proof}
The right inequality of \eqref{eq:m2} is  clearly sharp. In Section~\ref{sec4}, we will prove a stronger result,  improving the constant in the left inequality of \eqref{eq:m2} to $n(n-1)/4$. 

Following \citet{zhang2018note}, we write $\mathbb{E}_{i_1,\dots, i_k}$ for  expectation or average over all indices $1\leq i_1,\dots, i_k \leq n$, and $\widetilde{\mathbb{E}}_{i_1,\dots, i_k}$ for that over distinct indices $1 \leq i_1,\dots, i_k \leq n$.

\begin{theorem}[Recht--R\'e for $m=3$]\label{theorem 2}
Let $A_1,\dots, A_n$ be symmetric positive semidefinite and $A_1+\dots + A_n \preceq nI$. Then
\begin{equation}\label{eq:m3}
-I \preceq \widetilde{\mathbb{E}}_{i,j,k}A_iA_jA_k \preceq I.
\end{equation}
\end{theorem}
\begin{proof}
Let $A,B,C$ be positive semidefinite. Then $ABC+CBA\preceq ABA+CBC$. If $B\preceq C$, then $ABA\preceq ACA$.

We start with the right inequality of \eqref{eq:m3}, 
\begin{align*}
\widetilde{\mathbb{E}}_{i,j,k}A_iA_jA_k &= \frac{1}{2}\widetilde{\mathbb{E}}_{i,j,k}(A_iA_jA_k+A_kA_jA_i)\\
&\preceq \frac{1}{2}\widetilde{\mathbb{E}}_{i,j,k}(A_iA_jA_i+A_kA_jA_k)\\
& = \widetilde{\mathbb{E}}_{i,j,k} A_iA_jA_i.
\end{align*}

Fix a positive integer $l < n$ whose value we decide later.
\begin{align*}
\widetilde{\mathbb{E}}_{i,j,k}A_iA_jA_k  &\preceq \widetilde{\mathbb{E}}_{i,j,k} \Bigl[ \Bigl( 1-\frac{1}{l} \Bigr)A_iA_jA_k
+\frac{1}{l} A_iA_jA_i \Bigr]\\
&= \frac{1}{l^2(n-l)}\widetilde{\mathbb{E}}_{i_1,\dots, i_n}\bigl[(A_{i_1}+\dots+A_{i_l})\\
&\quad \cdot(A_{i_{l+1}}+\dots+A_{i_n})(A_{i_1}+\dots+A_{i_l})\bigr].
\end{align*}
Since $A_{i_{l+1}}+\dots+A_{i_n}\preceq nI-(A_{i_1}+\dots+A_{i_l})$,
\begin{align*}
\widetilde{\mathbb{E}}_{i,j,k}A_iA_jA_k \preceq \frac{1}{l^2(n-l)}\widetilde{\mathbb{E}}_{i_1,\dots ,i_l}\bigl[(A_{i_1}+\dots+A_{i_l})\\
\cdot\bigl(nI-(A_{i_1}+\dots+A_{i_l})\bigr)(A_{i_1}+\dots+A_{i_l})\bigr].
\end{align*}

Consider the function $f(x)=x^2(n-x)$. Let the line $y = cx+d$ be tangent to $f$ at $x = l$. We require that $c\geq 0$ and $f(x)\leq cx+d$ for $0\leq x \leq n$. Elementary calculation shows that such a line exists as long as $1/2\leq l/n\leq 2/3$. Let $A = A_{i_ 1} + \dots + A_{i_l}$. As $cA + dI$ and $A(nI - A)A$ are  simultaneous diagonalizable, each eigenvalue of $cA + dI -A(nI -A)A$ can be obtained by applying the function $g(x) = cx + d - x^2(n - x)$ to an eigenvalue of $A$. Hence
\begin{align*}
\widetilde{\mathbb{E}}_{i,j,k}&A_iA_jA_k \\
&\preceq \frac{1}{l^2(n-l)}\widetilde{\mathbb{E}}_{i_1,\dots ,i_l}\bigl[c(A_{i_1}+\dots+A_{i_l})+dI\bigr]\\
&\preceq \frac{cl+d}{l^2(n-l)} I,
\end{align*}
where the first inequality follows from the fact that it holds for each eigenvalue. Note that if we choose $A_1=\dots=A_n=I$, all inequalities above as well as the right inequality of \eqref{eq:m3} hold with equality. So as long as $1/2\leq l/n\leq 2/3$, $l, c, d$ will give us
\[
\frac{1}{l^2(n-l)}(cl+d) = 1
\]
and thus the right inequality of \eqref{eq:m3}.

For the left inequality of \eqref{eq:m3}, we start by noting
\[
(A_1+\dots+ A_{n-1})A_n(A_1+\dots+ A_{n-1})\succeq 0.
\]
Taking expectation, we have $-(n-2)\widetilde{\mathbb{E}}_{i,j,k}A_iA_jA_k\preceq \widetilde{\mathbb{E}}_{i,j,k}A_iA_jA_i$ and thus
\begin{align*}
-\widetilde{\mathbb{E}}_{i,j,k}&A_iA_jA_k \\
&= -\frac{n-2}{n-1}\widetilde{\mathbb{E}}_{i,j,k}A_iA_jA_k -\frac{1}{n-1}\widetilde{\mathbb{E}}_{i,j,k}A_iA_jA_k \\
&\preceq \frac{1}{n-1}\widetilde{\mathbb{E}}_{i,j,k}(A_iA_jA_i-A_iA_jA_k)\\
&\preceq \frac{1}{2(n-1)}\widetilde{\mathbb{E}}_{i,j,k}\bigl[(A_i-A_j)A_k(A_i-A_j)\bigr].
\end{align*}
As in the proof of Theorem~\ref{theorem1}, set $B \coloneqq nI-\sum_i A_i \succeq 0$ and $B_i \coloneqq A_i+\frac{1}{n}B \succeq 0$. Then
\begin{align*}
-\widetilde{\mathbb{E}}_{i,j,k}&A_iA_jA_k \\
&\preceq \frac{1}{2(n-1)}\widetilde{\mathbb{E}}_{i,j,k} \bigl[(B_i-B_j)B_k(B_i-B_j)\bigr]\\
&=\frac{1}{n-1}\widetilde{\mathbb{E}}_{i,j,k}B_iB_jB_i -\frac{1}{n-1}\widetilde{\mathbb{E}}_{i,j,k}B_iB_jB_k.
\end{align*}
Let $X_i \coloneqq B_{i}(nI-B_{i})B_{i}$ and $Y_i\coloneqq (nI-B_{i})B_{i}(nI-B_{i})$.
Routine calculations give
\begin{align*}
\widetilde{\mathbb{E}}_{i}X_i &= (n-1)\widetilde{\mathbb{E}}_{i,j,k}B_iB_jB_i,\\
\widetilde{\mathbb{E}}_{i}Y_i &= (n-1)\widetilde{\mathbb{E}}_{i,j,k}B_iB_jB_i\\
&\qquad\quad+(n-1)(n-2)\widetilde{\mathbb{E}}_{i,j,k}B_iB_jB_k,
\end{align*}
which allows us to express $\widetilde{\mathbb{E}}_{i,j,k}B_iB_jB_i$ and $\widetilde{\mathbb{E}}_{i,j,k}B_iB_jB_k$  in terms of $\widetilde{\mathbb{E}}_{i}X_i$ and $\widetilde{\mathbb{E}}_{i}Y_i$. Then
\begin{align*}
-\widetilde{\mathbb{E}}_{i,j,k}&A_iA_jA_k \\
&\preceq \frac{1}{n-1}\widetilde{\mathbb{E}}_{i,j,k}B_iB_jB_i -\frac{1}{n-1}\widetilde{\mathbb{E}}_{i,j,k}B_iB_jB_k\\
&=\frac{1}{(n-1)^2(n-2)}\widetilde{\mathbb{E}}_i[  (n-1)X_i-Y_i]\\
&=\frac{n}{(n-1)^2(n-2)}\widetilde{\mathbb{E}}_i[ -B_i(B_i-nI)(B_i-I) ].
\end{align*}
As $-x(x-n)(x-1)\leq (n-1)^2x/4$ for $0 \leq x \leq n$,
\[
-\widetilde{\mathbb{E}}_{i,j,k}A_i A_j A_k \preceq \frac{n}{4(n-2)}\widetilde{\mathbb{E}}_i B_i=\frac{n}{4(n-2)}I.
\]
When $n \geq 3$, we have  $n/\bigl(4(n-2)\bigr) \leq 1$.
\end{proof}
Our proof  in fact shows that the  constant in the left inequality of \eqref{eq:m3} can be improved to $n/\bigl(4(n-2)\bigr)$. Nevertheless, we will see in the next section (Table~\ref{table1}) that this is not sharp.

\section{Noncommutative Positivstellensatz}\label{sec3}

In a seminal paper \citep{helton2002positive}, Helton proved an astounding result: Every positive polynomial in noncommutative  variables can be written as a sum of squares of polynomials. The corresponding statement for usual polynomials, i.e., in commutative variables, is well-known to be false and is the subject of Hilbert's 17th Problem. Subsequent developments ultimately led to a noncommutative version of the Positivstellensatz  for semialgebraic sets. We refer interested readers to  \citet{Pascoe2018} for an overview of this topic.

Stating noncommutative Positivstellensatz will require that we introduce some terminologies. Let  $X_1,\dots, X_n$ be $n$ noncommutative variables, i.e., $X_i X_j \ne X_j X_i$ whenever $i \ne j$. A \emph{monomial} of degree $d$ or a \emph{word} of length $d$ is an expression of the form $X_{i_1} \cdots X_{i_d}$. The monomials span a real infinite-dimensional vector space $\mathbb{R}\langle X_1,\dots,X_n\rangle$, called the space of \emph{noncommutative polynomials}. For any $d \in \mathbb{N}$,  the finite-dimensional subspace of noncommutative polynomials of degree $\leq d$ will be denoted $\mathbb{R}\langle X_1,\dots,X_n\rangle_d$. The \emph{transpose} of $f \in \mathbb{R}\langle X_1,\dots,X_n\rangle$ is denoted  $f^\tp $ and is defined on monomials by reversing the order of variables  $(X_{i_1} \cdots X_{i_d})^\tp  = X_{i_d} \cdots X_{i_1}$ and extended linearly to all of $\mathbb{R}\langle X_1,\dots,X_n\rangle$. If $f^\tp  = f$, then $f$ is called \emph{symmetric}.

The bottom line is that noncommutative polynomials may be evaluated on square matrices of the same dimensions, i.e., they define matrix-valued functions of matrix variables. For our purpose, if $A_1,\dots,A_n$ are real symmetric matrices, then $f(A_1,\dots,A_n)$ is also a matrix, but it may not be a symmetric matrix unless $f$ is a symmetric polynomial.

Let $L = \{\ell_1,\dots, \ell_k\}\subseteq \mathbb{R}\langle X_1,\dots,X_n\rangle_1$ be a set of $k$ linear polynomials, i.e., $d=1$. We will refer to $\ell_1,\dots,\ell_k$ as \emph{linear constraints} and
\begin{multline*}
\mathcal{B}_L \coloneqq  \{(A_1,\dots,A_n) \mid 
\ell_1(A_1,\dots,A_n) \succeq 0, \dots,\\
 \ell_k(A_1,\dots,A_n) \succeq 0\}
\end{multline*}
as the \emph{feasible set}. Note that elements of $\mathcal{B}_L$ are $n$ tuples of symmetric matrices. We say that $\mathcal{B}_L$ is \emph{bounded} if there exists $r>0$ such that all $(A_1,\dots,A_n) \in \mathcal{B}_L$ satisfy $\lVert A_1\rVert\leq r,\dots,\lVert A_n\rVert \leq r$. Let $d \in \mathbb{N}$. We write
\[
\Sigma_d (L) \coloneqq  \biggl\{ \sum_{i=1}^k \sum_{j=1}^{p_i} f_{ij}^\tp  \ell_i f_{ij} \biggm|   \begin{multlined} f_{ij} \in \mathbb{R}\langle X_1,\dots,X_n\rangle_d, \\ k,p_1,\dots,p_k \in \mathbb{N} \end{multlined} \biggr\}
\]
for  the set of \emph{noncommutative sum-of-squares} generated by $L$. 
The following theorem is a simplified version of the noncommutative Positivstellensatz, i.e., Theorem~1.1 in \citet{helton2012convex}, that will be enough for our purpose. 
\begin{theorem}[Noncommutative Positivstellensatz]\label{thm:NP}
Let $f$ be a symmetric polynomial with $\degree(f)\leq 2d+1$ and the feasible set $\mathcal{B}_L$ be bounded  with nonempty interior. Then
\[
f(A_1,\dots,A_n) \succeq 0\; \text{ for all } \, (A_1,\dots,A_n) \in \mathcal{B}_L
\]
if and only if $f \in \Sigma_d(L)$.
\end{theorem}
Readers familiar with the commutative Positivstellsatz \citep{lasserre2015} would see that the noncommutative version is, surprisingly, much simpler and neater.

To avoid notational clutter, we introduce the shorthand
\[
\sum_{j_i \ne j_k} \coloneqq \sum_{\substack{1\leq j_1,\dots, j_m \leq n, \\j_1,\dots, j_m \text{ distinct}}} 
\]
for sum over distinct indices.
Applying Theorem~\ref{thm:NP} with linear constraints $X_1 \succeq 0,\dots,X_n \succeq 0$, $X_1 + \dots + X_n \preceq nI$, Conjecture~\ref{conjL} becomes the following.
\begin{conj}{1B}[Sum-of-squares form]\label{conjSOS}
Let $m \le n \in \mathbb{N}$ and  $d = \lfloor m/2 \rfloor$. For the linear constraints $\ell_1 = X_1, \dots, \ell_n=X_n, \ell_{n+1} = n-X_1 - \dots - X_n$, let
\begin{align*}
\lambda_1 &= \argmin\biggl\{\lambda \in \mathbb{R} \biggm| \lambda -\sum_{j_i \ne j_k} X_{j_1}\cdots X_{j_m} \in \Sigma_d (L) \biggr\},\\
\lambda_2 &= \argmin\biggl\{\lambda \in \mathbb{R} \biggm| \lambda +\sum_{j_i \ne j_k} X_{j_1}\cdots X_{j_m}  \in \Sigma_d (L)\biggr\}.
\end{align*}
Then both $\lambda_1$ and  $\lambda_2 \leq n!/(n-m)!$.
\end{conj}
In polynomial optimization \citep{lasserre2015}, the commutative Positivstellsatz is used to transform a constrained optimization problem into a sum-of-squares problem that can in turn be transformed into a semidefinite programming (SDP) problem. \citet{helton2002positive} has observed that this also applies to noncommutative polynomial optimization problems, i.e., we may further transform Conjecture~\ref{conjSOS} into an SDP form. 

The vector space $\mathbb{R}\langle X_1,\dots,X_n\rangle_d$  has dimension $q \coloneqq 1 + n + n^2 +\dots + n^d$ and a basis comprising all $q$ monomials of degree $\leq d$. We will assemble all basis elements into a $q$-tuple of monomials that we denote by $\beta$. With respect to this basis, any $f \in \mathbb{R}\langle X_1,\dots,X_n\rangle_d$ may be represented uniquely as $f = \beta^\tp  u$ for some $u\in \mathbb{R}^q$. Therefore a \emph{noncommutative square} may be expressed as
\[
\sum_{j=1}^{p} f_{j}^\tp  \ell f_{j} = \sum_{j=1}^{p} u_{j}^\tp  \beta \ell \beta^\tp  u_{j} = \tr\biggl[\beta \ell \beta^\tp \biggl(\sum_{j=1}^{p} u_{j} u_{j}^\tp \biggr)\biggr]
\]
by simply writing $f_{j} =  \beta^\tp  u_{j}$, $u_{j} \in \mathbb{R}^q$, $j = 1,\dots,p$. Since a symmetric matrix $Y$ is positive semidefinite iff it can be written as $Y =\sum_{j=1}^{p} u_{j} u_{j}^\tp$, we obtain the following one-to-one correspondence between noncommutative squares and positive semidefinite matrices:
\begin{gather*}
\sum_{j=1}^{p} f_{j}^\tp  \ell f_{j} \in \mathbb{R}\langle X_1,\dots,X_n\rangle_{2d+1}, \; f_{j} \in \mathbb{R}\langle X_1,\dots,X_n\rangle_d\\[-2.5ex]
\Big\Updownarrow\\[-2ex]
\sum_{j=1}^{p} u_{j} u_{j}^\tp  \in \mathbb{R}^{q\times q}, \; u_{j} \in \mathbb{R}^q.
\end{gather*}
With this correspondence, the two minimization problems in Conjecture~\ref{conjSOS} become two SDPs.
\begin{conj}{1C}[Semidefinite program form]\label{conjSDP}
Let $m \le n \in \mathbb{N}$ and  $d = \lfloor m/2 \rfloor$. Let $\beta$ be a monomial basis of $\mathbb{R}\langle X_1,\dots,X_n\rangle_d$ and let $X_{n+1} = n - X_1 - \cdots - X_n$.  Let $\lambda_1$ be the minimum value of the SDP:
\begin{equation}\label{eq:lambda1}
\begin{tabular}{rl}
\textnormal{minimize} & $\lambda$\\[-1.5ex]
\textnormal{subject to} & $\displaystyle\lambda - \sum_{j_i \ne j_k} X_{j_1}\cdots X_{j_m} = \sum_{i=1}^{n+1} \tr(\beta X_i \beta^\tp  Y_i)$,\\
& $Y_1 \succeq 0, \dots, Y_{n+1}\succeq 0$;
\end{tabular}
\end{equation}
and $\lambda_2$ be that of the SDP:
\begin{equation}\label{eq:lambda2}
\begin{tabular}{rl}
\textnormal{minimize} & $\lambda$\\[-1.5ex]
\textnormal{subject to} & $\displaystyle\lambda + \sum_{j_i \ne j_k} X_{j_1}\cdots X_{j_m} = \sum_{i=1}^{n+1} \tr(\beta X_i \beta^\tp  Y_i)$,\\
& $Y_1 \succeq 0, \dots, Y_{n+1}\succeq 0$.
\end{tabular}
\end{equation}
Then both $\lambda_1$ and  $\lambda_2 \leq n!/(n-m)!$.
\end{conj}
Note that the minimization is over the scalar variable $\lambda$ and the matrix variables $Y_1, \dots, Y_{n+1}$; the equality constraint equating two noncommutative polynomials is simply saying that the coefficients on both sides are equal, i.e.,  for each monomial, we get a \emph{linear constraint} involving $\lambda,Y_1,\dots,Y_{n+1}$ --- the $X_i$'s play no role  other than to serve as placeholders for these linear constraints. We may express \eqref{eq:lambda1} and \eqref{eq:lambda2} as SDPs in standard form with a single matrix variable $Y \coloneqq \diag(\lambda, Y_1, \dots, Y_{n+1})$, see \eqref{eq:primal} for example.

Readers acquainted with  (commutative) polynomial optimization \citep{lasserre2015} would be  familiar with the above discussions. In fact, the only difference between the commutative and noncommutative cases is that $\sum_{i=0}^d n^i$, the size of a noncommutative monomial basis, is much  larger than $\binom{d+n}{d}$, the size of a commutative monomial basis.

For any fixed values of $m$ and $n$, Conjecture~\ref{conjSDP} is in a form that can be checked by standard SDP solvers. The dimension of the SDP grows exponentially with $m$, and without access to significant computing resources, only small values of $m, n$ are within reach. Fortuitously, $m=n=5$ already yields the required violation $144.6488 \nleq 120$, showing that Conjecture~\ref{conjSDP}  and thus Conjecture~\ref{conj1} is false in general.  We tabulate our results for $m \le n \le 5$ in Table~\ref{table1}.

\begin{table}
\centering
\begin{tabular}{l||c|c|c}
 & $\lambda_1$ & $\lambda_2$ & $n!/(n-m)!$ \\ \hline\hline
 $m=2$, $n=2$ & 2.0000 & 0.5000 & 2   \\ \hline
 $m=2$, $n=3$ & 6.0000 & 1.5000 & 6  \\ \hline
 $m=2$, $n=4$ & 12.0000 & 3.0000 & 12 \\ \hline
 $m=2$, $n=5$ & 20.0000 & 5.0000 & 20 \\ \hline
 $m=3$, $n=3$ & 6.0000 & 3.4113 & 6 \\ \hline
 $m=3$, $n=4$ & 24.0000 & 8.5367 & 24 \\ \hline
 $m=3$, $n=5$ & 60.0000 & 17.3611 & 60 \\ \hline
 $m=4$, $n=4$ & 24.0000 & 22.4746  & 24  \\ \hline
 $m=4$, $n=5$ & 120.0000 & 80.2349 & 120 \\ \hline
 $m=5$, $n=5$ & 120.0000 & \textbf{144.6488} & 120 \\
\end{tabular}
\caption{Results from the SDPs in Conjecture~\ref{conjSDP} for $m \le n \le 5$. The bold entry for $\lambda_2$ shows that the Recht--R\'e conjecture is false for $m = n = 5$ since $144.6488 > 120$.  \label{table1}}
\end{table}

The fact that  the SDP in \eqref{eq:lambda2} for $m = n =5$ has a minimum $\lambda_2 > 144 > 120 = 5!$ shows that there are uncountably many instances with $A_1\succeq 0$, $A_2\succeq 0$, $A_3\succeq 0$, $A_4\succeq 0$, $A_5\succeq 0$, and $A_1+A_2+A_3+A_4+A_5\preceq 5I$ such that the matrix
\[
\sum_{\sigma \in \mathfrak{S}_5} A_{\sigma(1)}A_{\sigma(2)}A_{\sigma(3)}A_{\sigma(4)}A_{\sigma(5)}
\]
has an eigenvalue that is less than $-144 <  -120 = -5!$. Here $\mathfrak{S}_n$ is the symmetric group on $n$ elements. We emphasize that neither \eqref{eq:lambda2} nor its dual would give us  five such matrices explicitly, although the dual does provide another way to verify our result, as we will see in Section~\ref{sec:far}.

Indeed, the beauty of the noncommutative Positivstellensatz approach is that it allows us to show that  Conjecture~\ref{conj1} is false for $m=n=5$ without actually having to produce five positive semidefinite matrices $A_1, \dots, A_5$ that violates the inequality \eqref{eq:conj1}. It would be difficult to  find  $A_1, \dots, A_5$ explicitly as one does not even know the smallest dimensions required for these matrices to give a counterexample to \eqref{eq:conj1}. Our approach essentially circumvents the issue by replacing them with noncommutative variables $X_1,\dots,X_5$ --- the reader may have observed that the dimensions of the matrices $A_1,\dots,A_5$ did not make an appearance anywhere in this article.

\section{Verification via Farkas}\label{sec:far}

We take a closer look at the $m = n =5$ case that provided a refutation to the Recht--R\'e conjecture. In this case, the basis $\beta$ has $1+5+5^2 = 31$ monomials; the SDP in \eqref{eq:lambda2} has $1+5+5^2+5^3+5^4+5^5=3906$ linear constraints, $31^2\times 6+1 = 5767$ variables, and takes the form:
\begin{equation}\label{eq:primal}
\begin{tabular}{rl}
\textnormal{minimize} & $\tr(C_0 Y)$\\
\textnormal{subject to} & $\tr(C_i Y) =  b_i, \quad i =1,\dots,3906$,\\
& $Y =  \diag(\lambda, Y_1, \dots, Y_6) \succeq 0$.
\end{tabular}
\end{equation}
Here  $C_0,C_1,\dots,C_{3906} \in  \mathbb{S}^{187}_\pp$, $b \in \mathbb{R}^{3096}$,
 $\lambda$ is a scalar variable, and $Y_1,\dots,Y_6$ are $31$-by-$31$ symmetric matrix variables. To put \eqref{eq:primal} into standard form, the block diagonal structure of $Y$ may be further encoded as linear constraints requiring that off-diagonal blocks be zero. The output of our program gives a minimizer of the form $Y^* = \diag(\lambda^*, Y_1^*, \dots, Y_6^*) \in \mathbb{S}^{187}_\pp$ with
\begin{equation}\label{eq:values}
\lambda^*=144.6488,\quad Y_1^*,\dots, Y_6^* \in \mathbb{S}^{31}_\pp.
\end{equation}
The actual numerical entries of the matrices appearing in \eqref{eq:primal} and \eqref{eq:values} are omitted due to space constraints; but they can be found in the output of our program (code in supplement).

The values in \eqref{eq:values} are of course approximate because of the inherent errors in numerical computations. In our opinion, the gap between the computed $144.6488$ and the conjectured $120$ is large enough to override any concerns of a mistaken conclusion resulting from numerical errors. Nevertheless, to put to rest any lingering doubts, we will directly show that the conjectured value $\lambda = 120$ is infeasible by producing a Farkas certificate. Consider the feasibility problem:
\begin{equation}\label{eq:primal2}
\begin{tabular}{rl}
\textnormal{minimize} & $0$\\
\textnormal{subject to} & $\tr(C_i Y) =  b_i, \quad i =1,\dots,3906$,\\
& $\tr(C_0 Y) =120$,\\
& $Y \succeq 0$,
\end{tabular}
\end{equation}
with $C_0,C_1,\dots,C_{3906} \in  \mathbb{S}^{187}_\pp$ and $b \in \mathbb{R}^{3096}$ as in \eqref{eq:primal}. Note that $C_0 = e_1 e_1^\tp$ is the matrix with one in the $(1,1)$th entry and zero everywhere else. So \eqref{eq:primal2} is the feasibility problem of the optimization problem \eqref{eq:primal} with the additional linear constraint $y_{11}=120$ and where we have disregarded the block diagonal constraints\footnote{If \eqref{eq:primal2} is already infeasible, then adding these block diagonal constraints just makes it even more infeasible.}  on $Y$.
The dual of \eqref{eq:primal2} is
\[
\begin{tabular}{rl}
\textnormal{maximize} & $120 y_0 + b^\tp y$\\
\textnormal{subject to} & $y_0C_0 + y_1C_1 + \dots + y_{3906} C_{3906} \preceq 0$.
\end{tabular}
\]
Our program produces a Farkas certificate $y \in \mathbb{R}^{3096}$ with $120 y_0  + b^\tp y \approx 47.3 > 0$, implying that \eqref{eq:primal2} is infeasible. While this is a consequence of Farkas Lemma for SDP  \cite{farkas}, all we need is the following trivial version.
\begin{lemma}
Let $m,n \in \mathbb{N}$. Let $C_0,C_1,\dots,C_m \in \mathbb{S}^n$ and $b \in \mathbb{R}^{m+1}$.
If there exists a $y \in \mathbb{R}^{m+1}$ with
\[
y_0 C_0 + \dots + y_m C_m \preceq 0, \quad b^\tp y > 0,
\]
then there does not exist a $Y \in \mathbb{S}^n$ with
\[
\tr(C_0 Y) =  b_0,\dots,\tr(C_m Y) =  b_m,\quad Y \succeq 0.
\]
\end{lemma}
\begin{proof}
If such a $Y$ exists, then
\[
0 \ge \tr\bigl( (y_0 C_0 + \dots + y_m C_m ) Y \bigr) = y_0 b_0 + \dots + y_m b_m > 0,
\]
 a contradiction.
\end{proof}
Hence a matrix of the form
\[
Y = \diag(120, Y_1, \dots, Y_6) \in \mathbb{S}^{187}
\]
is \emph{infeasible} for \eqref{eq:primal}, providing another refutation of Conjecture~\ref{conjSDP} and thus Conjecture~\ref{conj1}. In particular, showing that $\lambda = 120$ is infeasible for \eqref{eq:primal} does not require any of the values computed in \eqref{eq:values}. Of course, aside from being the conjectured value of $\lambda_2$, there is nothing special about $\lambda = 120$ --- for any $\lambda < 144.6488$, we may similarly compute a Farkas certificate $y$ to show that such a value of $\lambda$ is infeasible for \eqref{eq:primal}.

We conclude with a few words on the computational costs of the SDPs in this and the last section.
Our resulting dense linear system for $m=n=5$ requires $3906\times 5767 \approx 22$ million floating point storage. Using a personal computer with an Intel Core i7-9700k processor and 16GB of RAM, our SeDuMi \cite{sturm1999using} program in Matlab takes $150$ seconds. For $m=n=6$, storage alone would have taken $26$ billion floating numbers, beyond our modest computing resources.

\section{Improving the Recht--R\'e inequality}\label{sec4}

An unexpected benefit of the noncommutative Positivstellensatz approach is that it leads to better bounds for the $m = 2$ and $3$ cases that we know are true. Observe that the values for $\lambda_2$ in Table~\ref{table1} for $m = 2$ are exactly smaller than the values for $n!/(n-m)!$ by a factor of $1/4$. This suggests that the Recht--R\'e inequality \eqref{eq:m2} for $m =2$ in Theorem~\ref{theorem1} may be improved to
\[
-\frac{1}{4}n(n-1)I \preceq \sum_{i\neq j}A_iA_j \preceq n(n-1)I.
\]
Table~\ref{table1} only shows this for $n = 2, 3, 4,5$ but in this section, we will give a proof for arbitrary $n \ge 2$. Although our proof below does not depend on the SDP formulation in \eqref{eq:lambda2}, the correct coefficients in \eqref{eq:coeff} for arbitrary $n$ would have been impossible to guess without solving \eqref{eq:lambda2} for $m =2$ and some small values of $n$.

So far we have not explored the symmetry evident in our formulations of the Recht--R\'e inequality: In Conjecture~\ref{conjL}, the matrix expression
\[
\lambda I \pm \sum_{j_i \ne j_k}  A_{j_1}\cdots A_{j_m}
\]
and the constraints $A_1 \succeq 0,\dots,A_n \succeq 0$, $A_1 + \dots + A_n \preceq nI$ are clearly invariant under any permutation $\sigma \in \mathfrak{S}_n$. In Conjecture~\ref{conjSDP}, the noncommutative sum-of-squares
\begin{equation}\label{eq:sos}
\lambda \pm \sum_{j_i \ne j_k}  X_{j_1}\cdots X_{j_m} = \sum_{i=1}^{n+1} \tr(\beta X_i \beta^\tp  Y_i),
\end{equation}
where $X_{n+1} = n - X_1 - \dots - X_n$,  is also invariant under $\mathfrak{S}_n$ and so we may average over all permutations to get a \emph{symmetrized sum-of-squares}. For commutative polynomials, results from classical invariant theory are often used to take advantage of symmetry \cite{Gatermann2004}. We will see next that  such symmetry may also be exploited for noncommutative polynomials.

Consider the case $m=2$, $n=3$. The monomial basis of $\mathbb{R}\langle X_1, X_2, X_3\rangle_1$ is $\beta = (1, X_1, X_2, X_3)$. The symmetry imposes linear constraints on  the matrix variables $Y_1, Y_2,Y_3, Y_4$ in \eqref{eq:lambda2}, requiring them to take the following forms:
\begin{align*}
Y_{1} &= 
 \begin{bmatrix}
  a & b & c & c \\
  b & d & e & e \\
  c & e & f & g  \\
  c & e & g & f
 \end{bmatrix}, & Y_{2} &= 
 \begin{bmatrix}
  a & c & b & c \\
  c & f & e & g \\
  b & e & d & e  \\
  c & g & e & f
 \end{bmatrix},\\
Y_{3} &= 
 \begin{bmatrix}
  a & c & c & b \\
  c & f & g & e \\
  c & g & f & e  \\
  b & e & e & d
 \end{bmatrix}, & Y_{4} &= 
 \begin{bmatrix}
  x & y & y & y \\
  y & z & w & w \\
  y & w & z & w \\
  y & w & w & z
 \end{bmatrix}.
\end{align*}
These symmetries allow us to drastically reduce the degree of freedom in our SDP: For any $m=2, n\geq 2$, the matrices $Y_1,\dots,Y_n$ are always determined by precisely $11$ variables that we label $a, b, c, d, e, f, g, x, y, z, w$. We computed their values explicitly for $n = 2, 3, 4$. 
For $n = 2$,
\[\renewcommand*{\arraystretch}{1.2}
Y_{1} = 
 \begin{bmatrix*}[r]
  \frac{5}{4} & -\frac{3}{4} & \frac{1}{4} \\
  -\frac{3}{4} & \frac{1}{2} & 0 \\
  \frac{1}{4} & 0 & \frac{1}{2}
 \end{bmatrix*}, \quad
Y_{3} = 
 \begin{bmatrix*}[r]
  \frac{1}{4} & -\frac{1}{4} & -\frac{1}{4} \\
  -\frac{1}{4} & \frac{1}{2} & 0 \\
  -\frac{1}{4} & 0 & \frac{1}{2}
 \end{bmatrix*},
\]
and $Y_2$ can be determined from $Y_1$.
For $n = 3$,
\[\renewcommand*{\arraystretch}{1.2}
Y_{1} = 
 \begin{bmatrix*}[r]
  \frac{5}{2} & -1 & 0 & 0 \\
  -1 & \frac{4}{9} & \frac{1}{9} & \frac{1}{9} \\
  0 & \frac{1}{9} & \frac{4}{9} & \frac{1}{9} \\
  0 & \frac{1}{9} & \frac{1}{2} & \frac{4}{9}
 \end{bmatrix*},
Y_{4} = 
 \begin{bmatrix*}[r]
  \frac{1}{2} & -\frac{1}{3} & -\frac{1}{3} & -\frac{1}{3} \\
  -\frac{1}{3} & \frac{4}{9} & \frac{1}{9} & \frac{1}{9} \\
  -\frac{1}{3} & \frac{1}{9} & \frac{4}{9} & \frac{1}{9} \\
  -\frac{1}{3} & \frac{1}{9} & \frac{1}{2} & \frac{4}{9}
 \end{bmatrix*},
\]
and $Y_2, Y_3$ can be determined from $Y_1$.
For $n = 4$,
\begin{align*}
Y_{1} &= \renewcommand*{\arraystretch}{1.2}
 \begin{bmatrix*}[r]
  \frac{15}{4} & -\frac{9}{8} & -\frac{1}{8} & -\frac{1}{8} & -\frac{1}{8} \\
  -\frac{9}{8} & \frac{3}{8} & \frac{1}{8} & \frac{1}{8} & \frac{1}{8}\\
  -\frac{1}{8} & \frac{1}{8} & \frac{3}{8} & \frac{1}{8} & \frac{1}{8}\\
  -\frac{1}{8} & \frac{1}{8} & \frac{1}{8} & \frac{3}{8} & \frac{1}{8}\\
  -\frac{1}{8} & \frac{1}{8} & \frac{1}{8} & \frac{1}{8} & \frac{3}{8}
 \end{bmatrix*},\\
Y_{5} &= \renewcommand*{\arraystretch}{1.2}
 \begin{bmatrix*}[r]
  \frac{3}{4} & -\frac{3}{8} & -\frac{3}{8} & -\frac{3}{8} & -\frac{3}{8} \\
  -\frac{3}{8} & \frac{3}{8} & \frac{1}{8} & \frac{1}{8} & \frac{1}{8}\\
  -\frac{3}{8} & \frac{1}{8} & \frac{3}{8} & \frac{1}{8} & \frac{1}{8}\\
  -\frac{3}{8} & \frac{1}{8} & \frac{1}{8} & \frac{3}{8} & \frac{1}{8}\\
  -\frac{3}{8} & \frac{1}{8} & \frac{1}{8} & \frac{1}{8} & \frac{3}{8}
 \end{bmatrix*},
\end{align*}
and $Y_2, Y_3, Y_4$ can be determined from $Y_1$. The rational numbers above are all chosen by observing the floating numbers output of the SDP \eqref{eq:lambda2}.

The values of the matrices $Y_i$'s for $n = 2,3,4$ allow us to guess that the variables  $a, b, c, d, e, f, g, x, y, z, w$ are:
\begin{equation}\label{eq:coeff}
\begin{gathered}
a = \frac{5(n-1)}{4},\quad b = -\frac{3(n-1)}{2n},\quad  c = \frac{3-n}{2n},\\
d = f = z = \frac{2(n-1)}{n^2}, \quad e = g = w = \frac{n-2}{n^2},\\
x = \frac{n-1}{4},\quad y = -\frac{n-1}{2n}.
\end{gathered}
\end{equation}
The proof of our next theorem will ascertain that these choices are indeed correct --- they yield the sum-of-squares decomposition in \eqref{eq:sos} for $m = 2$.

\begin{theorem}[Better Recht--R\'e for $m =2$]\label{thm:improve2}
Let $A_1, \dots, A_n$ be positive semidefinite matrices. If $A_1 + \dots + A_n \preceq n I$, then 
\[
-\frac{1}{4}n(n-1)I \preceq \sum_{i\neq j}A_i A_j  \preceq n(n-1)I .
\]
\end{theorem}
\begin{proof}
The upper bound has already been established in Theorem~\ref{theorem1}. It remains to establish the lower bound. We start from the following readily verifiable inequalities
\begin{align}\label{eq:a1}
\frac{1}{2n(n-1)}(A_j-A_k)A_i(A_j-A_k) &\succeq 0,\\
\frac{5(n-1)}{4}\Bigl(I-\frac{6}{5n}A_i+\frac{2(3-n)}{5n(n-1)}\sum_{j\neq i}A_j\Bigr)& \label{eq:a2} \\
A_i\Bigl(I-\frac{6}{5n}A_i+\frac{2(3-n)}{5n(n-1)}\sum_{j\neq i}A_j\Bigr) &\succeq 0, \nonumber\\
\frac{n-1}{5n^2}\Bigl(A_i+\frac{2n-1}{n-1}\sum_{j\neq i}A_j\Bigr)
& \label{eq:a3} \\
A_i\Bigl(A_i+\frac{2n-1}{n-1}\sum_{j\neq i}A_j\Bigr) &\succeq 0,\nonumber\\
\frac{1}{2n^2}(A_j-A_k)\Bigl(n- \sum_i A_i\Bigr)(A_j-A_k) &\succeq 0,\label{eq:a5}\\
\frac{n-1}{4}\Bigl(I-\frac{2}{n}\sum_i A_i\Bigr)\Bigl(n- \sum_i A_i\Bigr)& \label{eq:a4}\\
\Bigl(I-\frac{2}{n}\sum_i A_i\Bigr) &\succeq 0.\nonumber
\end{align}
Sum \eqref{eq:a1} over all distinct $i,j,k$; sum \eqref{eq:a2} over all $i$; sum \eqref{eq:a3} over all $i$; sum \eqref{eq:a5} over all distinct $j,k$; add all results to \eqref{eq:a4}. The final inequality is our required lower bound.
\end{proof}
For $n=m=2$, the new lower bound is sharp. Take
\[\renewcommand*{\arraystretch}{1.3}
A_{1} = 
\begin{bmatrix}
\frac{3}{2} & 0 \\
0 & 0
\end{bmatrix}, \qquad A_{2} = 
\begin{bmatrix}
\frac{1}{6} & \frac{\sqrt{2}}{3} \\
\frac{\sqrt{2}}{3} & \frac{4}{3}
\end{bmatrix},
\]
then $\lVert A_1+A_2\rVert = 2$ and the smallest eigenvalue of $A_1A_2+A_2A_1$ is $-1/2$. We conjecture that this bound is sharp for all $m=2$, $n\geq 2$.

The method in this section also extends to higher $m$. For example, we may impose symmetry constraints for $m=n=3$ and see if the $Y_1, Y_2, Y_3, Y_4$ obtained have rational values, and if so write down a sums-of-squares proof by factoring the $Y_i$'s.

\section{Conclusion and open problems}

We conclude our article with a discussion of some open problems and why we think the Recht--R\'e conjecture, while false as it is currently stated, only needs to be refined.

An immediate open question is whether the conjecture is true for $m = 4$: Table~\ref{table1} shows that it  holds for $(m,n) =(4,4)$ and $(4,5)$; we suspect that it is  true for all $n \ge 4$.

As we pointed out after Conjecture~\ref{conjL}, the Recht--R\'e inequality as stated in \eqref{eq:conj1} conceals an asymmetry --- it  actually contains two inequalities, as shown in \eqref{eq:conjL}. What we have seen is that the lower bound is never attained in any of the cases we have examined. For $m =2$ and $3$, the lower bound is too large, and we improved it in Theorem~\ref{thm:improve2} and the proof of Theorem~\ref{theorem 2} respectively. For $m = 5$, the lower bound is too small, which is why the Recht--R\'e inequality is false. A natural follow-up question is then: ``What is the correct lower bound?'' On the other hand, we conjecture that the remaining half of the Recht--R\'e inequality, i.e., the upper bound in \eqref{eq:conjL}, holds true for all $m \le n \in \mathbb{N}$.

\citet{recht2012toward} has another conjecture similar to Conjecture~\ref{conj1} but where the norms appear after the summation.
\begin{conj}{2}[\citealt{recht2012toward}] \label{conj2}
Let $A_1,  \dots, A_n$ be positive semidefinite matrices. Then
\begin{multline*}
\frac{1}{n^m}\sum_{1\leq j_1,\dots, j_m \leq n}\hspace*{-4ex} \lVert A_{j_1}\cdots A_{j_m}\rVert \geq \\ \frac{(n-m)!}{n!} \sum_{\substack{1\leq j_1,\dots, j_m \leq n, \\j_1,\dots, j_m \; \text{distinct}}} \hspace*{-4ex}\lVert A_{j_1}\cdots A_{j_m}\rVert .
\end{multline*}
\end{conj}
This has been established for $m=2$ and $3$ for any unitary invariant norm in \citet{israel2016arithmetic}. It is not clear to us if the  noncommutative Positivstellensatz might perhaps also shed light on this related conjecture.

Lastly, if our intention is to analyze the relative efficacies of with- and without-replacement sampling strategies in randomized algorithms, then it is more pertinent to study these inequalities for random matrices, i.e., we do not just assume that the indices are random variables but also the entries of the matrices. For example, if we want to analyze the Kaczmarz algorithm, then we ought to take expectation not only with respect to all permutations but also with respect to how we generate the entries of the matrices. This would provide a more realistic platform for comparing with- and without-replacement sampling strategies.


\bibliographystyle{icml2020}

\end{document}